\documentclass{article}

\usepackage{amsmath, amsthm, amssymb}

\usepackage{tikz, float}
\usetikzlibrary {positioning}

\author{Horia Mania}
\title{Wilmes' Conjecture and Boundary Divisors}
\date{\today}

\newtheorem{thm}{Theorem}
\newtheorem*{theorem}{Theorem}
\newtheorem{dfn}{Definition}
\newtheorem{rmk}{Remark}
\newtheorem{lemma}{Lemma}
\newtheorem{corollary}{Corollary}
\newtheorem{conjecture}{Conjecture}
\newtheorem{proposition}{Proposition}
\newtheorem*{conjecturess}{Conjecture}
\newtheorem*{propositionss}{Proposition}

\theoremstyle{definition}
\newtheorem{example}{Example}

\newcommand{\case}[1]{\subsubsection*{#1}}

\newcommand{\N}{\operatorname{N}}
\newcommand{\Lap}{\operatorname{L}}
\newcommand{\Cl}{\operatorname{Cl}}
\newcommand{\dist}{\operatorname{dist}}
\newcommand{\supp}{\operatorname{supp}}
\newcommand{\divisors}{\operatorname{div}}
\newcommand{\tG}{\operatorname{\widetilde{G}_\Pi}}
\newcommand{\AUS}{\operatorname{AUS(\tG)}}
\newcommand{\B}{\operatorname{B}}
\newcommand{\T}{\operatorname{T}}

\newcommand{\Pcal}{\operatorname{\mathcal{P}}}
\newcommand{\Jcal}{\operatorname{\mathcal{J}}}
\newcommand{\Acal}{\operatorname{\mathcal{A}}}
\newcommand{\Ecal}{\operatorname{\mathcal{E}}}
\newcommand{\Bcal}{\operatorname{\mathcal{B}}}

\newcommand{\tV}{\operatorname{\widetilde{V}_\Pi}}
\newcommand{\tE}{\operatorname{\widetilde{E}_\Pi}}
\newcommand{\Oo}{\operatorname{O}}
\newcommand{\ff}{\operatorname{f}}

\newcommand{\AO}{\operatorname{AO(\tG)}}
\newcommand{\img}{\operatorname{Img}}
\newcommand{\R}{\operatorname{R}}
\newcommand{\Hom}{\operatorname{\widetilde{H}}}

\newcommand\restr[2]{{
  \left.\kern-\nulldelimiterspace 
  #1 
  \vphantom{\big|} 
  \right|_{#2} 
  }}

\begin{document}

\tikzstyle{test}=[scale = .7,circle, fill = black!20, minimum size=3mm]
\tikzstyle{testB}=[scale = .7,circle, fill = black, minimum size=3mm]
\tikzstyle{testR}=[scale = .7,circle, fill = red, minimum size=3mm]
\tikzstyle{test2}=[scale = 1.5,circle, fill = black!0, inner sep = 0pt, minimum size=3mm]
\tikzstyle{test3}=[scale = 1.4,circle, fill = black!0, minimum size=3mm, inner sep = 0pt]
\tikzstyle{test4}=[scale = 1,circle, fill = black!0, inner sep = 0pt, minimum size=1pt]

\maketitle

\begin{abstract}
 This paper is concerned with Wilmes' conjecture regarding abelian sandpile models, as presented in \cite{W}. We introduce the concept of boundary divisors and use this to prove the conjecture for the first Betti number. Further results suggest that this method could be used to tackle the general case of the problem.   
\end{abstract}

\section{Introduction}

In this section we present Wilmes' conjecture and give an overview of the current work. The introduction to the problem mostly follows the exposition in \cite{W}. Familiarity with the sandpile model, also called the chip-firing model, is assumed. The reader new to the subject can refer to \cite{W} or \cite{PPW}. A more detailed presentation of the basic results concerning the chip-firing model can be found in \cite{AH}. In addition, the reader must be acquainted with Dhar's burning algorithm. It was first introduced in \cite{Dha}, but an exposition more suited for our purposes can be found in \cite{ASW}.

Throughout this paper, $G = (V,E)$ represents an undirected, connected graph with vertex set $V$ and edge set $E$. It is allowed to have multiple edges, but loops are disregarded. In order to state Wilmes' conjecture we need to introduce the Betti numbers of the homogeneous toppling ideal of the sandpile model. Usually, these are introduced by means of minimal free resolutions, and such a presentation can be found in both \cite{W} and \cite{PPW}. Related results concerning free resolutions of the toppling ideal, among others, appear in the paper by Manjunath and Sturmfels \cite{MS}. However, for our purposes it is enough to view the Betti numbers as the dimensions of the homology groups of a certain simplicial complex that we denote as $\Delta_D$.  See subsection \ref{background} for details. This is possible by Hochster's formula which we now state (\cite{PPW}, \cite{PS}): 

\begin{theorem}[Hochster's formula]
The $k$-th Betti number of a divisor $D$ is the dimension of the $(k-1)$-st reduced homology group of $\Delta_D$ over the complex numbers:

\[
\beta_{k,D} = \dim \Hom_{k-1} (\Delta_D, \mathbb{C}).
\]
\end{theorem}

The numbers $\beta_{k,D}$ represent refined information regarding the sandpile model. The conjecture is concerned with rougher information called the coarse Betti numbers: $\beta_k = \sum_{D \in \Cl(G)}\beta_{k,D}$. We will often refer to both of these sequences of numbers as Betti numbers. We are almost ready to state the problem, we just need to introduce some notation. We denote by $\Pcal_k$ the set of partitions of $G$ into $k$ connected components. For $\Pi \in \Pcal_k$ we denote by $G_\Pi = (V_\Pi, E_\Pi)$ the graph induced by the partition $\Pi$. The graph $\tG = (\tV, \tE)$ represents the simple graph induced by $G_\Pi$. Often we identify the vertices of $G_\Pi$ or $\tG$ with $\Pi$'s connected components. When we consider a connected component $\pi$ of $\Pi$, it will often be clear from the context if we view it as a vertex of $\tG$ or if we view it as a connected subset of vertices in $G$. Nonetheless, we use $\pi$ to refer to an element of $\tV$ and $\overline{\pi}$ to refer to the corresponding subset of $V$. 

\begin{conjecturess}[Wilmes \cite{W}]
Let $G$ be an undirected, connected graph. Then the $k$-th coarse Betti number associated to the graph $G$ satisfies:
\[ 
\beta_k = \sum_{\Pi \in \Pcal_{k+1}} \# \{G_\Pi \text{-maximal parking functions}\}.
\]
\end{conjecturess}

\noindent For a graph with $n$ vertices it is known that there are $n-1$ non-zero coarse Betti numbers \cite{PPW}. 


In this paper, we present a proof of the conjecture for the first Betti number (Theorem \ref{thmbeta1}), and provide new machinery that appears to be useful for studying higher Betti numbers. In particular, in the next section, we define the boundary divisors for cuts and  we use them to prove the conjecture for the first Betti number. This will justify the study of generalized boundary divisors in section \ref{generalized bd}. In addition, we give some more insight into the $(n-1)$-st case of the problem. Starting with section \ref{extension cycles}, we prepare the necessary tools to prove the relevance of the boundary divisors for multi-edged trees (these are trees in which we allow multiple edges). Namely, we prove a result concerning general Betti numbers $\beta_{k,D}$ in the case where $G$ is a multi-edged tree: 

\begin{propositionss}
If $G$ is a multi-edged tree and $D$ is a $\Pi$-boundary divisor for a partition $\Pi \in \Pcal_{k+1}$, then $\beta_{k,D} > 0$.
\end{propositionss}  

After completing this article, we learned of recent work by Mohammadi and Shokrieh \cite{FMFS}. They describe a minimal Gr\"obner basis for each higher syzygy module of the toppling ideal, giving an alternate proof of Wilmes' conjecture for all $k$. Additional independent work by Manjunath, Schreyer, and Wilmes \cite{MSW}, gives another proof of Wilmes conjecture by constructing minimal free resolutions via toric actions and Gr\"obner degenerations.

\subsection{Background Notions}
\label{background}

As mentioned before, we are concerned with an undirected, connected graph $G = (V,E)$. In general, we denote $\# V  = n$. One of the most important objects associated to a graph is the \emph{Laplacian}, denoted by $\Lap$. Given some ordering of $G$'s vertices, the Laplacian is a $n \times n$ matrix with entries $\{a_{ij}\}_{i,j}$, where $a_{ii}$ is equal to the degree of the $i$-th vertex, while the entries $a_{ij}$ for $i \neq j$ are equal to minus the number of edges connecting the vertices $i$ and $j$. 

We recall that a \emph{divisor} $D$ of $G$ is an element of the free group $\mathbb{Z}V$. In the context of the chip-firing model we denote $\divisors(G) = \mathbb{Z}V$. The group of divisors $\divisors(G)$ modulo the image of the Laplacian represents the class group of $G$, denoted by $\Cl(G)$. Two divisors $D_0$ and $D_1$ are said to be equivalent if they belong to the same class in $\Cl(G)$ and this is denoted by $D_0 \sim D_1$. Since $G$ is connected, the eigenspace corresponding to the eigenvalue $\lambda = 0$ of the Laplacian has dimension equal to one. This immediately implies that if $D_0 \sim D_1$, there exists a unique \emph{script} $\sigma \in \mathbb{Z}V$ such that $\sigma \geq 0$, $\sigma \not > 0 $ and $D_0 -\Lap \sigma = D_1$. The reader might remark that divisors and scripts are essentially the same, they are both elements of the free group $\mathbb{Z}V$. We choose to make a distinction between them because a divisor encodes the number of chips present at each vertex, while a script encodes frequencies of vertex firings.

If $D_0 \in \divisors(G)$, the set $|D_0| = \{D_1 \in \divisors(G)|\text { } D_0 \sim D_1 \text{ and } D_1 \geqslant 0\}$ is called the linear system of $D_0$ and represents the set of non-negative divisors equivalent to $D_0$. A non-negative divisor is sometimes called \textit{effective}. The degree of $D$ is equal to $\sum_{v\in V} D_v$, the sum of its components. The \textit{support} of a divisor $D$ is the set $\supp(D) = \{v\in V|D_v \neq 0\}$. We now turn to the mysterious simplicial complex mentioned above. For a divisor $D \in \divisors(G)$ consider the simplicial complex $\Delta_D = \{W|\text{ } W \subset \supp(D_0) \text{ for some } D_0\in |D| \}$.
 
An important point to be made is that Wilmes' phrased his conjecture in the language of "recurrent configurations". We do not follow this approach here. The interested reader can refer to \cite{PPW}. For our purposes, the notion of a $G$-parking function is more useful. A divisor $D$ is said to be $G$-parking with respect to a sink $s \in V$ if there is no script $\sigma$ such that $s \not \in \supp(\sigma)$ and $D -\Lap \sigma \geq 0$. $G$-parking functions are sometimes called superstable configurations. Let us denote by $1_v$ the divisor $D$ that assigns one chip to the vertex $v$ and zero to the other vertices. By a maximal $G$-parking function, we mean a divisor $D$ that is $G$-parking and has the property that $D + 1_v$ is not $G$-parking $\forall v\in V$.   

\section{The First Betti Number}
In this section, the focus is on proving Wilmes' conjecture for the first Betti number. This represents the starting point for the study of boundary divisors. Before we begin, we remark that it is enough to prove that the first Betti number is equal to the number of distinct cuts of the graph. This is because when $k=1$, each graph $G_\Pi$ is a connected graph with exactly two vertices, and thus admits exactly one maximal parking function. 

\begin{thm}
\label{thmbeta1}
The first Betti number of the minimal free resolution for the toppling ideal is equal to the number of cuts in a graph. 
\end{thm}

The main idea of the proof is to construct a correspondence between $\Pi \in \Pcal_2$ and divisors $D$ with $\beta_{1,D} > 0$. More precisely, the argument is twofold. On one hand we show that for certain classes of cuts there are certain boundary divisors with $\beta_{1,D} > 0$ to which they correspond. On the other hand if $\beta_{k,D} = l >0$, we show that $D$ is a boundary divisor and that there are exactly $l$ distinct cuts that produce boundary divisors equivalent to it. We first introduce some concepts necessary for what follows. 

\begin{dfn}
Given $A,B \subset V$ two sets of vertices of $G$, we define the \emph{crossing-degree} of a vertex $v \in V$ with respect to $A$ and $B$ as
\[
\deg_{AB}(v) = \# \{ e \in E \colon e \text{ incident to v and } e \text{ an edge between A and B} \}.
\]
\end{dfn}

\begin{dfn}
Let $\Pi \in \Pcal_2$ and let $A$ and $B$ be its two connected components. A divisor $D$ is a \emph{boundary divisor} with respect to 
$\Pi$ if there exists $X \in \{A,B\}$ such that:

\[
D_v = \deg_{AB}(v)\chi_X(v), \quad \forall v \in V.
\]
\end{dfn}

\noindent Recall that $\chi_X(v)$ is the characteristic function of the set $X$, defined to be $1$ if $v \in X$, and $0$ otherwise. 

\begin{rmk}
Let $A$ and $B$ be as before. If $D(A)$ and $D(B)$ are two boundary divisors such that $D(A)_v = \deg_{AB}(v)\chi_A(v)$ and $D(B)_v = \deg_{AB}(v)\chi_B(v)$, then it is easy to check that $D(A) \sim D(B)$ (to obtain $D(B)$ from $D(A)$, just fire all the vertices in $A$ exactly once). Hence, up to equivalence there is a unique boundary divisor that corresponds to a cut $\Pi$. This will not remain true for boundary divisors with respect to a general partition $\Pi \in \Pcal_k$ (for $k \geq 3$). 
\end{rmk}

\begin{example}
\label{example1}
Let $G$ be a graph with vertices $\{a,b,c,d\}$ as represented in the picture below. The dotted lines represent two cuts $\Pi_1$ and $\Pi_2$, in Figure $1$ and Figure $2$, respectively. The numbers written inside the vertices represent numbers of chips corresponding to the boundary divisors of $G$ with respect to $\Pi_1$ and $\Pi_2$. For each of the two cuts both of its boundary divisors are represented in the figures.

\begin{figure}[ht]
\label{fig1}
\begin{center}

\begin{tikzpicture}
[scale = .3, very thick = 15mm]
  \node(m1) at (2.8,23.3) [test2] {$a$};
  \node(m2) at (0.3,19.3) [test2] {$b$};
  \node(m3) at (7.6,19.2) [test2] {$c$};
  \node(m4) at (2.8,13.3) [test2] {$d$};
    
      \node (n4) at (4,13)  [test] {4};
  \node (n1) at (4,23) [test] {0};
  \node (n2) at (1,18)  [test] {0};
  \node (n3) at (7,18)  [test] {0};
  \foreach \from/\to in {n1/n2,n1/n3, n3/n4,n2/n4}
    \draw[] (\from) -- (\to);
    \path[] (n2) edge [bend left = 15] node {} (n3);
    \path[] (n2) edge [bend right = 15] node {} (n3);
    \path[] (n3) edge [bend left = 15] node {} (n4);
    \path[] (n3) edge [bend right = 15] node {} (n4);
    \draw[dotted] (1,15.5) -- (7,15.5);
    
      \node (n4) at (14,13)  [test] {0};
  \node (n1) at (14,23) [test] {0};
  \node (n2) at (11,18)  [test] {1};
  \node (n3) at (17,18)  [test] {3};
  \foreach \from/\to in {n1/n2,n1/n3, n3/n4,n2/n4}
    \draw[] (\from) -- (\to);
    \path[] (n2) edge [bend left = 15] node {} (n3);
    \path[] (n2) edge [bend right = 15] node {} (n3);
    \path[] (n3) edge [bend left = 15] node {} (n4);
    \path[] (n3) edge [bend right = 15] node {} (n4);
    \draw[dotted] (11,15.5) -- (17,15.5);
\end{tikzpicture}

\caption{The two $\Pi_1$-boundary divisors of the graph $G$.}
\end{center}

\end{figure}
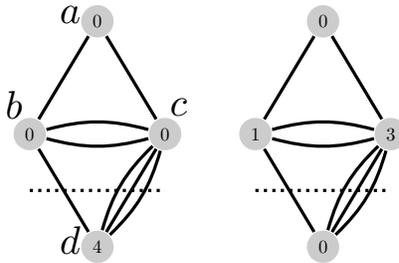

\end{example}

\begin{figure}[H]

\label{fig2}
\begin{center}

\begin{tikzpicture}
[scale = .3, very thick = 15mm]
  
  \node (n4) at (4,1)  [test] {3};
  \node (n1) at (4,11) [test] {0};
  \node (n2) at (1,6)  [test] {3};
  \node (n3) at (7,6)  [test] {0};
  \foreach \from/\to in {n1/n2,n1/n3, n3/n4,n2/n4}
    \draw[] (\from) -- (\to);
    \path[] (n2) edge [bend left = 15] node {} (n3);
    \path[] (n2) edge [bend right = 15] node {} (n3);
    \path[] (n3) edge [bend left = 15] node {} (n4);
    \path[] (n3) edge [bend right = 15] node {} (n4); 
   \draw[dotted] (1.5,11) -- (6.8,1);
   
    \node (n4) at (14,1)  [test] {0};
  \node (n1) at (14,11) [test] {1};
  \node (n2) at (11,6)  [test] {0};
  \node (n3) at (17,6)  [test] {5};
  \foreach \from/\to in {n1/n2,n1/n3, n3/n4,n2/n4}
    \draw[] (\from) -- (\to);
    \path[] (n2) edge [bend left = 15] node {} (n3);
    \path[] (n2) edge [bend right = 15] node {} (n3);
    \path[] (n3) edge [bend left = 15] node {} (n4);
    \path[] (n3) edge [bend right = 15] node {} (n4);
    \draw[dotted] (11.5,11) -- (16.8,1);
   
\end{tikzpicture}

\caption{The two $\Pi_2$-boundary divisors of the graph $G$.}
\end{center}

\end{figure}
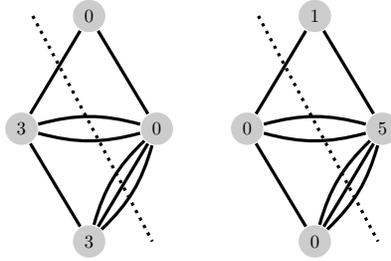

We need to introduce the concept of \emph{neighbouring vertices} of a set $A\subset V$: $\N(A) = \{v \in V \colon \dist(v, A) = 1\}$. 

\begin{lemma}
Let $D$ be one of the boundary divisors of a cut $\Pi$ with connected components $A$ and $B$. Then there is no effective divisor $D_0$ equivalent to $D$ and supported in both $A$ and $B$. 
\end{lemma}

\begin{proof}
We can assume without loss of generality that $D = D(A)$, that is $\supp (D) \subset A$. Suppose there exists an effective divisor $D_0 \sim D$ such that $\supp (D_0) \cap A \neq \emptyset$ and $\supp (D_0) \cap B \neq \emptyset$. There is a unique script $\sigma \geqslant 0$ and $\sigma \not > 0$ such that $D_0 - \Lap \sigma = D$. Recall that $\sigma$ encodes frequencies of vertex firings. By assumption $\supp(D_0) \cap B \neq \emptyset$. Hence $\supp(\sigma) \cap B$ is non-empty because $\supp(D) \cap B$ is empty. Let $v \in \supp(\sigma) \cap B$. Then all vertices in $\N(v) \cap B$ belong to $\supp(\sigma)$, again because $\supp(D) \cap B = \emptyset$. Since $B$ is a connected component we can argue inductively to conclude that $B \subset \supp(\sigma)$. Let $\restr{\sigma}{B}$ be the script that satisfies $(\restr{\sigma}{B})_v = \sigma_v$ when $v \in B$ and $(\restr{\sigma}{B})_v$ is zero otherwise. Also, let $D_1 = D_0 -\Lap \restr{\sigma}{B}$. By the way we defined the boundary divisor $D$ and since $B \subset \supp (\sigma)$, it follows that $(D_1)_v \geq D_v$ for all $v \in A$. Now, because $\sigma = \restr{\sigma}{A} + \restr{\sigma}{B}$, a similar argument with the one above shows that $A \subset \supp(\sigma)$. However, this contradicts $\sigma \not > 0$. This completes the proof of the lemma. 
\end{proof}

\begin{dfn}
Let $\Delta_D$ be the simplicial complex associated to some divisor $D \in \Cl(G)$, with $\beta_{1,D} > 0$. We say that $D_0 \in |D|$ is a \emph{composing divisor} of a connected component $C$ of $\Delta_D$ if $\supp(D_0) \subset C$. 
\end{dfn}

\begin{dfn}

Consider the same situation as in the previous definition. We say that a \emph{splitting} of $D \in \Cl(G)$ is a non-trivial partition of $|D|$ into two disjoint sets $\Acal$ and $\Bcal$, with the additional property that if $C \subset \Delta_D$ is a connected component of the simplicial complex, its composing divisors are contained in the same partition set of the splitting. (composing divisors of different connected components can belong to distinct partition sets)
\end{dfn}

\begin{rmk}
We require $\beta_{1,D} > 0$ because otherwise $D$ would not admit a splitting. This is because the dimension of the zero-th homology group of a simplicial complex is equal to the number of distinct, connected components of the complex minus $1$. 
\end{rmk}

\begin{example}
In order to get a better sense of what a splitting and a composing divisor are, let us consider again the graph $G$ and the partition $\Pi_1$ presented in Figure \ref{fig1}. Also, let $D$ be the $\Pi_1$-boundary divisor that has $4$ chips at the vertex $d$. Then, it can be checked that $|D| = \{0004, 2020, 0130\}$, where the $i$-th digit represents the number of chips on the $i$-th vertex in alphabetical order. This means that the simplicial complex $\Delta_D$ can be represented as follows:

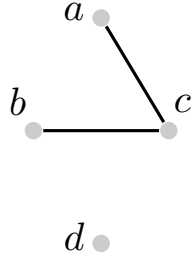
\begin{figure}[H]
\label{fig3}

\begin{center}

\begin{tikzpicture}
[scale = .3, very thick = 15mm]
  
  \node(m1) at (2.8,23.3) [test2] {$a$};
  \node(m2) at (0.3,19.3) [test2] {$b$};
  \node(m3) at (7.6,19.2) [test2] {$c$};
  \node(m4) at (2.8,13.3) [test2] {$d$};
    
      \node (n4) at (4,13)  [test] {};
  \node (n1) at (4,23) [test] {};
  \node (n2) at (1,18)  [test] {};
  \node (n3) at (7,18)  [test] {};
  \foreach \from/\to in {n1/n3, n2/n3}
    \draw[] (\from) -- (\to);
   
\end{tikzpicture}

\caption{The simplicial complex $\Delta_D$ for the $\Pi_1$-boundary divisor $D$.}
\end{center}
\end{figure}

From Figure $3$ it is immediate that $\Delta_D$ has $2$ connected components: $\{a,b,c\}$ and $\{d\}$. The composing divisors of $\{a,b,c\}$ are $2020$ and $0130$, and the only possible splitting of $D$ is $|D| = \{2020,0130\}\sqcup\{0004\}$. The divisor $0004$ is the only composing divisor of $\{d\}$.
\end{example}

\begin{lemma}
\label{splitting lemma}
Let $D \in \Cl(G)$ with $\beta_{1,D} > 0$. Then, for any splitting $|D| = \Acal \sqcup \Bcal$, there exist $D_A \in \Acal$ and $D_B \in \Bcal$ such that the set of edges between $S_A = \supp(D_A)$ and $S_B = \supp(D_B)$ represents a cut of the graph. Moreover, the boundary divisors generated by this cut are equal to $D$ in $\Cl(G)$.  
\end{lemma}

\begin{proof}
We can assume $D \in \Bcal$ and we choose $D_0 \in \Acal$. They are equivalent so there exists a unique script $\sigma \geq 0$, such that $\sigma \ngtr 0$ and $D_0 - \Lap \sigma = D$. Pick $s \in V \setminus \supp(\sigma)$ and set it to be the sink. Then in the new reduced setting $D_0$ and $D$ are still equivalent. Moreover, there is a unique $G$-parking function equivalent to both of them. Let $D_G$ be that $G$-parking function. We consider $D_G$ with the appropriate number of chips on the sink so that it is equivalent to $D_0$ and $D$ in the unreduced setting. Without loss of generality suppose that $D_G \in \Bcal$.

We now apply repeatedly Dhar's burning algorithm. Let $D_k$ be the divisor obtained after applying the algorithm $k$ times to $D_0$, and let $S_k$ denote its support. Let $t \geq 0$ such that $D_t \in \Acal$ and $D_{t+1} \in \Bcal$ (there is such a $t$ since $D_0 \in \Acal$ and $D_G \in \Bcal$). Since $D_t$ and $D_{t+1}$ belong to different connected components of $\Delta_D$, we have that $S_t \cap S_{t+1} = \emptyset$. Furthermore, we get that $S_{t} \subset \N(S_{t+1})$ and $S_{t+1} \subset \N(S_t)$ because the burning algorithm at one iteration can move a chip at most distance one. Let $E_t$ be the set of edges between $S_t$ and $S_{t+1}$ (it is clearly non-empty by the previous remark).

We now notice that removing $E_t$ disconnects the graph. When applying Dhar's algorithm for the $(t+1)$-st time, all vertices in $S_t$ remain unburned. This is needed since $S_t$ and $S_{t+1}$ are disjoint. If the graph were still connected after removing $E_t$, then we could consider a shortest path from $s$ to some vertex of $S_t$ that avoids the vertices in $S_{t+1}$. This vertex will have an incident burned edge and it will hence, after firing, send some chips to a vertex that is not in $S_{t+1}$. We thus obtain a contradiction, so the graph disconnects when removing $E_t$.

As we noticed that all the vertices in $S_t$ have to be unburned, it is easy to see that all vertices in $S_{t+1}$ have to burn. Hence, if $v \in S_t$ has $d_v$ incident edges going to $S_{t+1}$, then $(D_t)_v = d_v$. Suppose there exists a connected component of the graph after removing $E_t$ such that it contains only a subset $S \subsetneq S_t$, but not the entire $S_t$. Then, by firing every vertex in this connected component exactly once we send some chips from $S_t$ to $S_{t+1}$, but not all of them. We thus obtain a divisor that is connected in $\Delta_D$ to both $S_t$ and $S_{t+1}$. This is a contradiction, hence all the vertices of $S_t$ are contained in the same connected component. An analogous argument shows that $S_{t+1}$ is contained in one connected component, and now it is easy to see that $E_t$ is a cut. The arguments above prove the second part of the lemma as well. The proof of the lemma is complete.
\end{proof}

One more notion is needed in the proof of Theorem \ref{thmbeta1}. We say two cuts, $\Pi_1$ and $\Pi_2$, \emph{intersect} if we cannot label their connected components $A_1$, $B_1$ and $A_2$, $B_2$ such that $A_1 \subset A_2$ and $B_2 \subset B_1$.

\begin{proof}[Proof of Theorem \ref{thmbeta1}]
The idea of the proof consists in finding a correspondence between cuts of the graph and divisors $D \ge 0$ with $\beta_{1,D} > 0$. This correspondence is provided by the two lemmas presented above. 

We first prove that $\beta_1 \geq \# \Pcal_2$. We call two cuts equivalent if their boundary divisors are equivalent. Let $\Pi_1$ and $\Pi_2$ be two different equivalent cuts. The first thing to notice is that these cuts cannot intersect themselves. If this were not true, we would get that each of the boundary divisors of one of the cuts is supported on both sides of the other cut. But this is in contradiction with the first lemma. We now group the cuts of the graph into equivalence classes. Suppose an equivalence class contains $k$ cuts. Then, by applying lemma \ref{prop2} and the non-intersecting property of the cuts, we get that the simplicial complex of a boundary divisor for one of the cuts has at least $k+1$ connected components. Hence, $\beta_{1,D}$ for this divisor is at least $k$. We thus obtain that $\beta_1 \geq \# \Pcal_2$.

Finally, we are left to prove that $\beta_1 \leq \# \Pcal_2$. Take $D\in \Cl(G)$ with $\beta_{1,D}  = l > 0$ and let $|D| = C_1 \sqcup C_2 \sqcup ... \sqcup C_{l+1}$, where $C_i$ represents the set of divisors composing the $i$-th connected component of $\Delta_D$. We represent graphically the $l+1$ connected components by one vertex each and we connect $C_i$ to $C_j$ ($i\neq j$) by an edge if there exist $D_0 \in C_i$ and $D_1 \in C_j$ a pair of equivalent boundary divisors. By lemma \ref{splitting lemma} we get that the graph formed in this way is connected. Hence, it has at least $l$ edges. Lemma \ref{splitting lemma} also guarantees that distinct elements of $\Cl(G)$ will produce distinct cuts, we then obtain that $\beta_1 \leq \# \Pcal_2$. The proof of the theorem is complete.  
\end{proof}

\begin{example}
We now illustrate in a particular case the results obtained above.  We consider the graph $G$ introduced in Example \ref{example1}. The first Betti number for this graph can be computed (with the sandpile package \cite{Per} of the mathematics software Sage \cite{SG}, for example) to be $\beta_1 = 6$. By inspection, it can be seen that the graph $G$ admits exactly six cuts. The cuts are presented in Figure $4$, together with their corresponding boundary divisors. Again by inspection or through the use of the software Sage, the reader can check that for each of these boundary divisors $D$ the first Betti number $\beta_{1,D}$ is equal to one, and that there are no two equivalent boundary divisors produced by distinct cuts (in general, this need not be the case). The simplicial complexes of the boundary divisors are represented in Figure $5$.

\begin{figure}[ht]

\label{figure thm1}
\begin{center}

\begin{tikzpicture}
[scale = .22, very thick = 15mm]
  
  \node (n4) at (4,1)  [test] {3};
  \node (n1) at (4,11) [test] {0};
  \node (n2) at (1,6)  [test] {3};
  \node (n3) at (7,6)  [test] {0};
  \foreach \from/\to in {n1/n2,n1/n3, n3/n4,n2/n4}
    \draw[] (\from) -- (\to);
    \path[] (n2) edge [bend left = 15] node {} (n3);
    \path[] (n2) edge [bend right = 15] node {} (n3);
    \path[] (n3) edge [bend left = 15] node {} (n4);
    \path[] (n3) edge [bend right = 15] node {} (n4); 
   \draw[dotted] (1.5,11) -- (6.8,1);
   
    \node (n4) at (14,1)  [test] {0};
  \node (n1) at (14,11) [test] {1};
  \node (n2) at (11,6)  [test] {3};
  \node (n3) at (17,6)  [test] {0};
  \foreach \from/\to in {n1/n2,n1/n3, n3/n4,n2/n4}
    \draw[] (\from) -- (\to);
    \path[] (n2) edge [bend left = 15] node {} (n3);
    \path[] (n2) edge [bend right = 15] node {} (n3);
    \path[] (n3) edge [bend left = 15] node {} (n4);
    \path[] (n3) edge [bend right = 15] node {} (n4);
    \draw[dotted] (11.5,1) -- (16.8,11);

    \node (n4) at (24,1)  [test] {4};
  \node (n1) at (24,11) [test] {0};
  \node (n2) at (21,6)  [test] {0};
  \node (n3) at (27,6)  [test] {0};
  \foreach \from/\to in {n1/n2,n1/n3, n3/n4,n2/n4}
    \draw[] (\from) -- (\to);
    \path[] (n2) edge [bend left = 15] node {} (n3);
    \path[] (n2) edge [bend right = 15] node {} (n3);
    \path[] (n3) edge [bend left = 15] node {} (n4);
    \path[] (n3) edge [bend right = 15] node {} (n4);
    \draw[dotted] (21,3) -- (27,3);
   
\end{tikzpicture}

\bigskip

\begin{tikzpicture}
[scale = .22, very thick = 15mm]

  \node (n4) at (4,1)  [test] {0};
  \node (n1) at (4,11) [test] {2};
  \node (n2) at (1,6)  [test] {0};
  \node (n3) at (7,6)  [test] {0};
  \foreach \from/\to in {n1/n2,n1/n3, n3/n4,n2/n4}
    \draw[] (\from) -- (\to);
    \path[] (n2) edge [bend left = 15] node {} (n3);
    \path[] (n2) edge [bend right = 15] node {} (n3);
    \path[] (n3) edge [bend left = 15] node {} (n4);
    \path[] (n3) edge [bend right = 15] node {} (n4); 
   \draw[dotted] (1,8.5) -- (7,8.5);
   
    \node (n4) at (14,1)  [test] {0};
  \node (n1) at (14,11) [test] {0};
  \node (n2) at (11,6)  [test] {4};
  \node (n3) at (17,6)  [test] {0};
  \foreach \from/\to in {n1/n2,n1/n3, n3/n4,n2/n4}
    \draw[] (\from) -- (\to);
    \path[] (n2) edge [bend left = 15] node {} (n3);
    \path[] (n2) edge [bend right = 15] node {} (n3);
    \path[] (n3) edge [bend left = 15] node {} (n4);
    \path[] (n3) edge [bend right = 15] node {} (n4);
    \draw[dotted] (12.5,1) -- (12.5,11);
    
    \node (n4) at (24,1)  [test] {0};
  \node (n1) at (24,11) [test] {0};
  \node (n2) at (21,6)  [test] {0};
  \node (n3) at (27,6)  [test] {6};
  \foreach \from/\to in {n1/n2,n1/n3, n3/n4,n2/n4}
    \draw[] (\from) -- (\to);
    \path[] (n2) edge [bend left = 15] node {} (n3);
    \path[] (n2) edge [bend right = 15] node {} (n3);
    \path[] (n3) edge [bend left = 15] node {} (n4);
    \path[] (n3) edge [bend right = 15] node {} (n4);
    \draw[dotted] (25.5,1) -- (25.5, 11);
   
\end{tikzpicture}

\caption{The six cuts of the graph $G$ and the boundary divisors they produce in $\Cl(G)$.}
\end{center}

\end{figure}
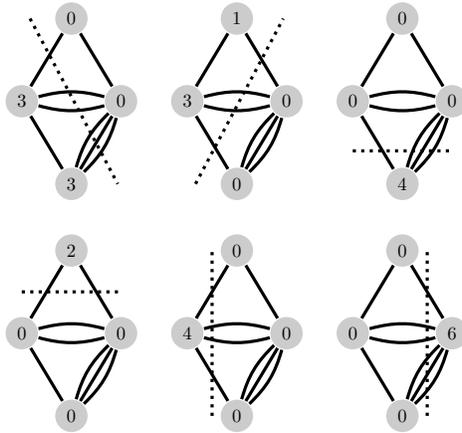

\begin{figure}[H]

\label{figure2 thm1}
\begin{center}

\begin{tikzpicture}
[scale = .22, very thick = 15mm]
 
  \node (n4) at (4,1)  [test] {};
  \node (n1) at (4,11) [test] {};
  \node (n2) at (1,6)  [test] {};
  \node (n3) at (7,6)  [test] {};
  \foreach \from/\to in {n1/n3, n2/n4}
    \draw[] (\from) -- (\to);
   
    \node (n4) at (14,1)  [test] {};
  \node (n1) at (14,11) [test] {};
  \node (n2) at (11,6)  [test] {};
  \node (n3) at (17,6)  [test] {};
  \foreach \from/\to in {n1/n2, n3/n4}
    \draw[] (\from) -- (\to);
    
    \node (n4) at (24,1)  [test] {};
  \node (n1) at (24,11) [test] {};
  \node (n2) at (21,6)  [test] {};
  \node (n3) at (27,6)  [test] {};
  \foreach \from/\to in {n3/n2,n1/n3}
    \draw[] (\from) -- (\to);

\end{tikzpicture}

\bigskip

\begin{tikzpicture}
[scale = .22, very thick = 15mm]
  
  \node (n1) at (4,11) [test] {};
  \node (n2) at (1,6)  [test] {};
  \node (n3) at (7,6)  [test] {};
  \foreach \from/\to in {n3/n2}
    \draw[] (\from) -- (\to);
   
   \fill[black!10] (14,11) -- (14,1) -- (17,6) -- cycle;
   
    \node (n4) at (14,1)  [test] {};
  \node (n1) at (14,11) [test] {};
  \node (n2) at (11,6)  [test] {};
  \node (n3) at (17,6)  [test] {};
  \foreach \from/\to in {n1/n3, n3/n4, n1/n4}
    \draw[] (\from) -- (\to);
    
    \fill[black!10] (24,11) -- (24,1) -- (21,6) -- cycle;
    
    \node (n4) at (24,1)  [test] {};
  \node (n1) at (24,11) [test] {};
  \node (n2) at (21,6)  [test] {};
  \node (n3) at (27,6)  [test] {};
  \foreach \from/\to in {n1/n2,n2/n4, n1/n4}
    \draw[] (\from) -- (\to);

\end{tikzpicture}

\caption{The six simplicial complexes of the boundary divisors presented in Figure $4$.}
\end{center}

\end{figure}
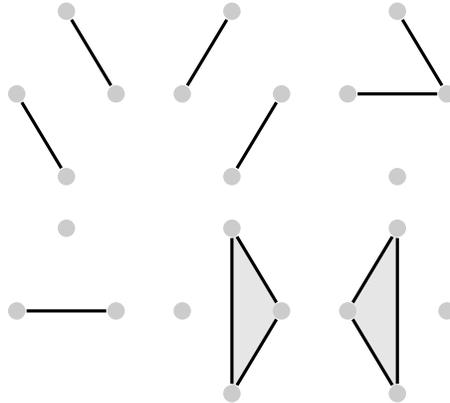

\end{example}
\section{Boundary divisors}
\label{generalized bd}

This section introduces the notion of boundary divisors for a partition $\Pi \in \Pcal_{k+1}$. These special divisors have nice properties that make them a promising way to tackle Wilmes' conjecture as we saw in the last section. Moreover, at the end of this section we will see how boundary divisors describe the structure behind the last Betti number. 

Before we define the boundary divisors we need to introduce some notions. For $k \geqslant 1$, we say that a \emph{generating sequence} of $\Pi \in \Pcal_{k+1}$ is a sequence $\{C_i\}_{i = 1}^k$ of cuts, where $C_1$ is a cut of $G$ and $C_i$, for $i \geqslant 2$, is a cut for one of the connected components produced by $C_{i-1}$, such that when applied sequentially the sequence of cuts produces the partition $\Pi$. One can easily convince himself that any partition $\Pi \in \Pcal_{k+1}$ admits at least one generating sequence (one just needs to check that for any graph $\tG$ there exits a vertex $v \in \tV$ such that $\tG \setminus v$ is connected). In general, if $\{C_i\}_{i = 1}^k$ is a generating sequence for some $\Pi \in \Pcal_{k+1}$, we denote by $A_i$ and $B_i$ the two connected components produced by $C_i$.

\begin{example}
\label{example3}
Let us consider again the graph $G$ introduced in Example \ref{example1} and consider the partition $\Pi$ with $3$ connected components: $\{a\}$, $\{c\}$ and $\{b,d\}$. This partition has three generating sequences, which can be seen in Figure $6$.

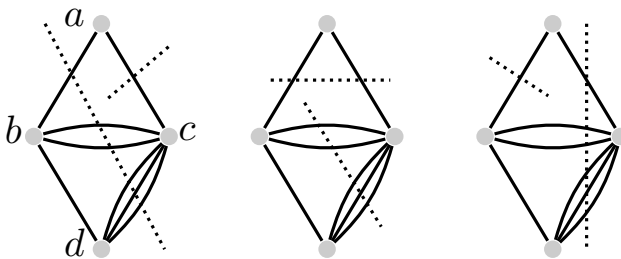
\begin{figure}[ht]
\label{figure4}
\begin{center}

\begin{tikzpicture}
[scale = .3, very thick = 15mm]
  
  \node(m1) at (2.8,11.3) [test2] {$a$};
  \node(m2) at (0.1,6.3) [test2] {$b$};
  \node(m3) at (7.8,6.2) [test2] {$c$};
  \node(m4) at (2.8,1.3) [test2] {$d$};
  
   \node (n4) at (4,1)  [test] { };
  \node (n1) at (4,11) [test] {};
  \node (n2) at (1,6)  [test] {};
  \node (n3) at (7,6)  [test] {};
  \foreach \from/\to in {n1/n2,n1/n3, n3/n4,n2/n4}
    \draw[] (\from) -- (\to);
    \path[] (n2) edge [bend left = 15] node {} (n3);
    \path[] (n2) edge [bend right = 15] node {} (n3);
    \path[] (n3) edge [bend left = 15] node {} (n4);
    \path[] (n3) edge [bend right = 15] node {} (n4); 
   \draw[dotted] (1.5,11) -- (6.8,1);
   \draw[dotted] (4.3,7.7) -- (7,10);
   
    \node (n4) at (14,1)  [test] {};
  \node (n1) at (14,11) [test] {};
  \node (n2) at (11,6)  [test] {};
  \node (n3) at (17,6)  [test] {};
  \foreach \from/\to in {n1/n2,n1/n3, n3/n4,n2/n4}
    \draw[] (\from) -- (\to);
    \path[] (n2) edge [bend left = 15] node {} (n3);
    \path[] (n2) edge [bend right = 15] node {} (n3);
    \path[] (n3) edge [bend left = 15] node {} (n4);
    \path[] (n3) edge [bend right = 15] node {} (n4);
    \draw[dotted] (11.5,8.5) -- (16.8,8.5);
    \draw[dotted] (13,7.5) -- (16.4, 2);
    
     \node (n4) at (24,1)  [test] {};
  \node (n1) at (24,11) [test] {};
  \node (n2) at (21,6)  [test] {};
  \node (n3) at (27,6)  [test] {};
  \foreach \from/\to in {n1/n2,n1/n3, n3/n4,n2/n4}
    \draw[] (\from) -- (\to);
    \path[] (n2) edge [bend left = 15] node {} (n3);
    \path[] (n2) edge [bend right = 15] node {} (n3);
    \path[] (n3) edge [bend left = 15] node {} (n4);
    \path[] (n3) edge [bend right = 15] node {} (n4);
    \draw[dotted] (25.5,11) -- (25.5,1);
    \draw[dotted] (21.2,9.5) -- (23.8, 7.8);
    
\end{tikzpicture}

\caption{The three generating sequences of the partition $\Pi$.}
\end{center}

\end{figure}

\end{example}

In this section we make the convention that $k$ is an integer between $1$ and  $n-1$ ($n = \# V$ as before) and that $\Pi$ is an element of $\Pcal_{k+1}$. We are now ready to define the boundary divisors.

\begin{dfn}
A divisor $D \in \divisors(G)$ is a \emph{boundary divisor} if there exists a generating sequence $\{C_i\}_{i=1}^k$ for some $\Pi \in \Pcal_{k+1}$, and if there exists a choice of $X_i \in \{A_i, B_i\}$ such that:
\begin{equation}
\label{bd}
D_v = \sum_{i = 1}^k \deg_{A_iB_i}(v)\chi_{X_i}(v), \quad \forall v \in V.
\end{equation} 
A $\Pi$-boundary divisor is a divisor that is boundary with respect to $\Pi$.
\end{dfn}

From this definition, the reader can easily see that a $\Pi$-boundary divisor is supported on the "boundaries" of $\Pi$'s connected components.

Before we arrive at the first important result concerning boundary divisors we have to establish a correspondence between $\Pi$-boundary divisors and orientations of the graph $\tG$. This relation will allow us to easily study equivalence relations concerning boundary divisors. We begin by introducing a map $\ff \colon \Oo(\tG) \rightarrow \divisors(\tG)$ from the set of orientations of $\tG$ to the set of divisors on $G$. For $o \in \Oo (\tG)$ the map $\ff$ is defined as follows:
\[
(\ff(o))_v = \sum_{ \substack{ e \in \tE(o)\\
                               v \in \overline{e^+}}} 
\deg_{\overline{e^-} \text{ } \overline{e^+}}(v) \text{, for } v \in V.
\]

\begin{example}
\label{example4}
We consider the partition $\Pi$ and take its first generating sequence from Example \ref{example3}. With respect to this partition the divisor $D$ presented in Figure $7$ is boundary. We also look at at the graph $\tG$ and the orientation $o$ that can be seen in the picture below. Computing $\ff (o)$, one can see that $\ff (o) = D$.

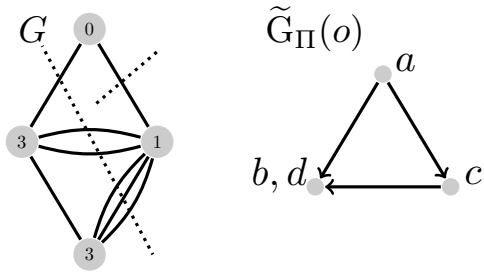
\begin{figure}[ht]
\label{figure5}
\begin{center}

\begin{tikzpicture}
[scale = .3, very thick = 15mm]
  
   \node (n4) at (4,1)  [test] {3};
  \node (n1) at (4,11) [test] {0};
  \node (n2) at (1,6)  [test] {3};
  \node (n3) at (7,6)  [test] {1};
  \foreach \from/\to in {n1/n2,n1/n3, n3/n4,n2/n4}
    \draw[] (\from) -- (\to);
    \path[] (n2) edge [bend left = 15] node {} (n3);
    \path[] (n2) edge [bend right = 15] node {} (n3);
    \path[] (n3) edge [bend left = 15] node {} (n4);
    \path[] (n3) edge [bend right = 15] node {} (n4); 
   \draw[dotted] (1.5,11) -- (6.8,1);
   \draw[dotted] (4.3,7.7) -- (7,10);
   
   \node (m2) at (1.5,11) [test3] {$G$};
    \node (m1) at (14,11) [test3] {$\tG (o)$};
    \node (m3) at (18,9.5) [test2] {$a$};
    \node (m4) at (12.4,4.5) [test2] {$b,d$};
    \node (m5) at (21,4.5) [test2] {$c$};
  \node (n1) at (17,9) [test] {};
  \node (n2) at (14,4)  [test] {};
  \node (n3) at (20,4)  [test] {};

    \path[<-] (n2) edge  node {} (n3);
    \path[<-] (n3) edge  node {} (n1);
    \path[<-] (n2) edge  node {} (n1);

\end{tikzpicture}

\caption{A $\Pi$-boundary divisor $D$ of $G$ and the graph $\tG$ with an orientation $o$ such that $\ff (o) = D$.}
\end{center}
\end{figure}
\end{example}

It is not a coincidence that in the previous example $\ff(o)$ equals the boundary divisor $D$. We now study a sufficient condition for an orientation of $\tG$ to be mapped by $\ff$ to a boundary divisor. Fix a vertex $s \in \tV$ to be the sink of the sandpile model on $\tG$ and the source of the acyclic orientations with unique source on $\tG$. Recall that $s$ can be viewed at the same time as a vertex of $\tG$ and as a connected subset of vertices of $G$, and that when we refer to it as the latter, we denote it by $\overline{s}$. We denote the set of acyclic orientations with unique source $s$ by $\AUS$. We check that the restriction $\restr{\ff}{\AUS}$ is a map from $\AUS$ to $\B(\Pi)$, the set of $\Pi$-boundary divisors on $G$. Let $o \in \AUS$, we denote by $\tG(o) = (\tV(o), \tE(o))$ the graph $\tG$ with orientation $o$. 

\begin{lemma}
The restriction $\restr{\ff}{\AUS}$ is a map from $\AUS$ to $\B(\Pi)$.
\end{lemma}

\begin{proof}
Choose $o\in \AUS$, we show that $\ff(o) \in \B(\Pi)$.  It is easy to verify that a simple graph with an acyclic orientation has at least one source and at least one target. Hence, there exists a target $t \in \tV(o)$. We show that removing all edges contained in $\tE(o)$ and incident to $t$ represents a cut of $\tG(o)$. Suppose this were not true. Then, after removing the edges incident to $t$ there would exist at least two distinct connected components of the graph $\tG(o)$ that do not contain $t$. Each of these connected components has an induced acyclic orientation and thus each of them contains a source. But since all of  $\tG(o)$'s edges incident to $t$ point toward $t$ it follows that $\tG(o)$ has to contain at least two sources, which is a contradiction with $o\in \AUS$. This means we can take the first element of $\Pi$'s generating sequence to be the cut that has $\overline{t}$ as one of the connected components, and choose $X_1 = \overline{t}$ in (\ref{bd}). Now we can look at $\tG(o) \setminus t$ and proceed inductively by removing a target at each step. This shows that $\ff(o) \in \B(\Pi)$ and hence $\ff$ is a well defined map from $\AUS$ to $\B(\Pi)$.
\end{proof}

The number of maximal $G_\Pi$-parking functions is equal to the number of $\tG$-maximal parking functions (the  sink is $s$ for both graphs). The reader can convince himself of this fact by using Dhar's burning algorithm. Furthermore, the number of maximal $\tG$-parking functions is equal to $\# \AUS$, the number of acyclic orientations on $\tG$ with unique source $s$ \cite{PPW}. This fact allows us to show that the number of distinct $\Pi$-boundary divisors in $\Cl(G)$ is equal to the number of maximal $G_\Pi$-parking functions. Namely, we show that the map $\ff \colon \AUS \rightarrow \B(\Pi) / \img \Lap$ is a bijection.

\begin{example}
Before we arrive at the precise statement of the last claim, let us take a second look at the boundary divisor from Example $\ref{example4}$. If $\ff$ is a bijection, it follows that up to equivalence there are exactly two $\Pi$-boundary divisors of the graph $G$. One of them is the one we already saw in Example \ref{example4}, while the other one can be found in the picture below. By inspection, it is immediate that $o$ and $o'$ are the only two orientations with unique source $a$ of $\tG$. 

\begin{figure}[ht]
\label{figure6}
\begin{center}

\begin{tikzpicture}
[scale = .3, very thick = 15mm]
  
   \node (n4) at (4,1)  [test] {0};
  \node (n1) at (4,11) [test] {0};
  \node (n2) at (1,6)  [test] {1};
  \node (n3) at (7,6)  [test] {6};
  \foreach \from/\to in {n1/n2,n1/n3, n3/n4,n2/n4}
    \draw[] (\from) -- (\to);
    \path[] (n2) edge [bend left = 15] node {} (n3);
    \path[] (n2) edge [bend right = 15] node {} (n3);
    \path[] (n3) edge [bend left = 15] node {} (n4);
    \path[] (n3) edge [bend right = 15] node {} (n4); 
   
   \node (m2) at (1.5,11) [test3] {$G$};
    \node (m1) at (14,11) [test3] {$\tG (o')$};
    \node (m3) at (18,9.5) [test2] {$a$};
    \node (m4) at (12.4,4.5) [test2] {$b,d$};
    \node (m5) at (21,4.5) [test2] {$c$};
  \node (n1) at (17,9) [test] {};
  \node (n2) at (14,4)  [test] {};
  \node (n3) at (20,4)  [test] {};

    \path[->] (n2) edge  node {} (n3);
    \path[<-] (n3) edge  node {} (n1);
    \path[<-] (n2) edge  node {} (n1);

\end{tikzpicture}

\caption{A $\Pi$-boundary divisor $D'$ of $G$ and the graph $\tG$ with the orientation $o'$ such that $\ff (o') = D'$.}
\end{center}
\end{figure}
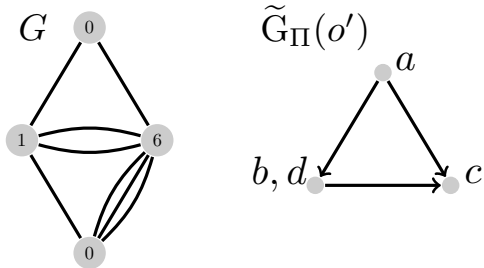 
\end{example}

\begin{proposition}
\label{prop1}
The map $\ff \colon \AUS \rightarrow \B(\Pi) / \img \Lap$ is a bijection. Hence, the number of distinct $\Pi$-boundary divisors in $\Cl(G)$ is equal to the number of maximal $G_\Pi$-parking functions.
\end{proposition}

\begin{proof}
We begin by proving that $\ff$ is surjective. Let $D \in \B(\Pi)$ be a boundary divisor, then we want to show that there exists $o \in \AUS$ such that $D \sim \ff(o)$. First we prove that there exists an acyclic orientation $o_1 \in \Oo(\tG)$ such that $\ff(o_1) = D$. Let $\{C_i\}_{i=1}^k$ be a $\Pi$-generating sequence and let $\{X_i\}_{i=1}^k$ be a choice of sets that can produce $D$ as defined in (\ref{bd}). We construct $o_1$ in the following way: if $\pi_1,\pi_2 \in \tV$ are connected by an edge, then there exists $i\in \overline{1,k}$ such that $C_i$ cuts some edges between $\overline{\pi}_1$ and $\overline{\pi}_2$ in $G$. In this case, $o_1$ orients $\tG$'s edge between $\pi_1$ and $\pi_2$ toward the connected component contained in $X_i$. For this to be a well defined orientation it must be checked that if $C_i$ cuts an  edge between $\overline{\pi}_1$ and $\overline{\pi}_2$, then $C_i$ cuts all the edges between these two connected components. Let $C_i$ be the first cut that removes an edge between them. By the definition of a generating sequence we get that both $\overline{\pi}_1$ and $\overline{\pi}_2$ are contained in one of the connected components produced by $C_{i-1}$, call this component $A_{i-1}$. Then $C_i$ is a cut of $A_{i-1}$. But because $C_i$ removes at least one edge between $\overline{\pi}_1$ and $\overline{\pi}_2$, it follows that $C_i$ removes all paths between $\overline{\pi}_1$ and $\overline{\pi}_2$ in $A_{i-1}$. Hence, $C_i$ removes all the edges between $\overline{\pi}_1$ and $\overline{\pi}_2$. By the definitions of $D$ and $\ff$, one can remark that $D = \ff(o_1)$. 

We now verify that $o_1$ is acyclic. Suppose 
\[
\pi_1 \rightarrow \pi_2 \rightarrow ... \rightarrow \pi_r \rightarrow \pi_1
\]
is a cycle of $o_1$. Then, choose $i$ minimal such that $C_i$ cuts one of the edges connecting $\overline{\pi}_j$ and $\overline{\pi}_{j+1}$, where $1 \leqslant j \leqslant r$. We make the convention $\pi_{r+1} = \pi_1$. By the minimality of $i$, it follows that $\overline{\pi}_j \subset A_{i-1}$ for all $j \in \overline{1,r}$. Without loss of generality suppose $C_i$ cuts an edge between $\overline{\pi}_1$ and $\overline{\pi}_2$.  By the above discussion regarding the definition of $o_1$, we obtain that $\overline{\pi}_1 \subset A_i$ and $\overline{\pi}_2 \subset B_i$ (where we might have to interchange the names of $A_i$ and $B_i$). Since $C_i$ removes all the paths between $\overline{\pi}_1$ and $\overline{\pi}_2$, it follows that there must exist $2 \leqslant l \leqslant r$ such that $\overline{\pi}_l \subset B_i$ and $\overline{\pi}_{l+1} \subset A_i$.
Since $o_1$ orients the edge between $\pi_1$ and $\pi_2$ towards $\pi_2$, it follows that $X_i = B_i$, in equation (\ref{bd}). But this implies that in $\tG(o_1)$ the edge between $\pi_l$ and $\pi_{l+1}$ is oriented toward $\pi_l$. This is in contradiction with the choice of the cycle $o_1$. Hence, $o_1$ is acyclic.

Let $u \in \Oo(\tG)$. For $v$ and $w$ two vertices in $\tV (u)$ we say that $w$ is \emph{reachable} from $v$ if there is a directed path from $v$ to $w$; we allow the empty path. Let $\R(u)$ denote the vertices reachable from $s$ in $\tG(u)$. If $\R(o_1) = \tV$, it is immediate that $s$ is the unique source of $o_1$ and thus we obtained $o_1 \in \AUS$ such that $\ff(o_1) \sim D$. Otherwise, by the choice of the set $\R(o_1)$ we know that all the edges between $\R(o_1)$ and $\tV \setminus \R(o_1)$ are directed towards $\R(o_1)$. We consider the orientation $o_2 \in \Oo(\tG)$ that orients the edges between the two sets just mentioned away from $\R(o_1)$, and orients the rest of the edges in the same direction as $o_1$ does. One can verify that $o_2$ is acyclic as well. We now look at the divisor obtained from $\ff(o_1)$ by firing all vertices in $\overline{\R(o_1)} = \bigcup_{v \in \R(o_1)}\overline{v}$ exactly once. It is easy to check that this divisor is effective and is equal to $\ff(o_2)$. We consider the set $\R(o_2)$ and repeat the procedure till we obtain $\R(o_l) = \tV$. This procedure will terminate since $\R(o_i) \subsetneq \R(o_{i+1})$ (equality holds only when $\R(o_i) = \tV$). Then $o_l \in \AUS$ such that $D \sim \ff(o_l)$. Hence, $\ff \colon \AUS \rightarrow \B(\Pi)/\img(L)$ is surjective.    

We are left to prove that $\ff$ is injective. Let $o_1$ and $o_2$ be two distinct elements of $\AUS$. We need to show $\ff(o_1) \not \sim \ff(o_2)$. Suppose they are equivalent, then there exists a unique script $\sigma \geqslant 0$, $\sigma \not > 0$ such that $\ff(o_1) - \Lap\sigma = \ff(o_2)$. We first show that $\supp(\sigma) \cap \overline{s} = \emptyset$.

Suppose there exists $v \in \supp(\sigma) \cap \overline{s}$. Then $\N(v) \cap \overline{s} \subset \supp(\sigma) \cap \overline{s}$ because otherwise there exists $w \in \N(v) \cap \overline{s}$ such that $(\ff(o_1) -\Lap\sigma)_w > \ff(o_2)_w = 0$. By applying this argument inductively and using that $\overline{s}$ is a connected component, it follows that $\overline{s} \subset \supp(\sigma)$. Let $s_1 = s$ and pick $s_2$ a source of $\tG(o_2) \setminus s_1$. Since $\supp(f(o_1)) \cap \overline{s_2} \neq \emptyset$ and since $\overline{s} \subset \supp(\sigma)$, by the choice of $s_2$ it follows that $\exists$ $v \in \overline{s_2} \cap \supp(\sigma)$. By the choice of $s_2$ we obtain similarly as before  that $\N(v) \cap \overline{s_2} \subset \supp(\sigma)$. Since $\overline{s_2}$ is connected in $G$, it follows as above that $\overline{s_2} \subset \supp(\sigma)$. We now argue inductively that $\supp(\sigma) = V$ by choosing at each step a source $s_l$ of the graph $\tG(o_1) \setminus \{s_1, s_2,...,s_{l-1}\}$. However, $\supp(\sigma) \neq V$ since $\sigma \not > 0$. Thus $\supp(\sigma) \cap \overline{s} = \emptyset$ as claimed. 

In order to complete the proof of the proposition we colour the edges of  $\tG$ as follows: if an edge has the same orientation in both $o_1$ and $o_2$, we colour it in black. Otherwise we colour it in red. We call a vertex in $\tV$ safe if all its incident edges are black. Since $o_1, o_2 \in \AUS$, it is immediate that $s$ is safe. We now colour the vertices of $\tG$ in the following way: if a vertex is safe and there exists a undirected path of safe vertices from $s$ to it, we colour it in black. Otherwise we colour it in red. By construction the set of black vertices in $\tG$ is connected. 

\begin{example}
Before completing the proof of the proposition, we consider an example of a graph $\tG$ and two acyclic orientations of it with unique source in order to illustrate how the edges and vertices are coloured. It is clear from the picture that $d$ is the only safe vertex that is not black. 
 
 
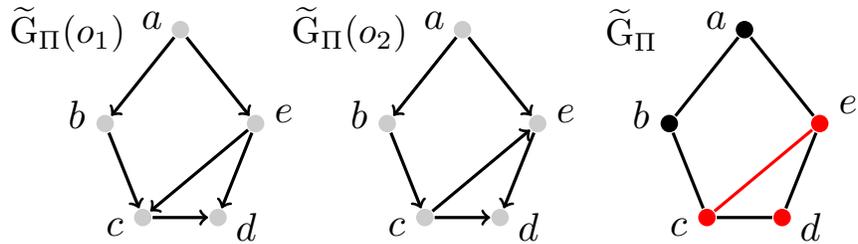
\begin{center}
 \begin{tikzpicture}
[scale = .25, very thick = 15mm]
  
  \node (n1) at (5,11) [test] {};
  \node (n2) at (1,6)  [test] {};
  \node (n3) at (9,6)  [test] {};
  \node (n4) at (3,1)  [test] {};
  \node (n5) at (7,1)  [test] {};
  \foreach \from/\to in {n1/n2,n2/n4, n1/n3,n3/n4, n3/n5, n4/n5}
    \draw[->] (\from) -- (\to);
   
   \node (m2) at (-1,11) [test3] {$\tG(o_1)$};
    \node (m1) at (14,11) [test3] {$\tG (o_2)$};
  
  \node (n1) at (20,11) [test] {};
  \node (n2) at (16,6)  [test] {};
  \node (n3) at (24,6)  [test] {};
  \node (n4) at (18,1)  [test] {};
  \node (n5) at (22,1)  [test] {};
  \foreach \from/\to in {n1/n2,n2/n4, n1/n3, n4/n3, n3/n5, n4/n5}
    \draw[->] (\from) -- (\to);
    
    \node(p1) at (3.5, 11.5) [test2] {$a$};
    \node(p2) at (-0.5, 6.5) [test2] {$b$};
        \node(p3) at (10.5, 6.5) [test2] {$e$};
    \node(p4) at (1.5, 0.5) [test2] {$c$};
    \node(p5) at (8.5, 0.5) [test2] {$d$};

	\node(p1) at (18.5, 11.5) [test2] {$a$};
    \node(p2) at (14.5, 6.5) [test2] {$b$};
        \node(p3) at (25.5, 6.5) [test2] {$e$};
    \node(p4) at (16.5, 0.5) [test2] {$c$};
    \node(p5) at (23.5, 0.5) [test2] {$d$};

  \node (n1) at (35,11) [testB] {};
  \node (n2) at (31,6)  [testB] {};
  \node (n3) at (39,6)  [testR] {};
  \node (n4) at (33,1)  [testR] {};
  \node (n5) at (37,1)  [testR] {};
  \foreach \from/\to in {n1/n2,n2/n4, n1/n3, n3/n5, n4/n5}
    \draw[] (\from) -- (\to);
    \draw[red] (n3) -- (n4);
   
   \node (m2) at (29,11) [test3] {$\tG$};

    \node(p1) at (33.5, 11.5) [test2] {$a$};
    \node(p2) at (29.5, 6.5) [test2] {$b$};
    \node(p3) at (40.5, 7) [test2] {$e$};
    \node(p4) at (31.5, 0.5) [test2] {$c$};
    \node(p5) at (38.5, 0.5) [test2] {$d$};

\end{tikzpicture}
\end{center}

\end{example}

We return to the proof. Transfer $\tG$'s colouring to $\tG(o_1)$ and $\tG(o_2)$ in the obvious way. We look at the graph induced by $\tG(o_2)$ on the set of red vertices, denote it by $\R^{r}_\Pi (o_2)$. This induced graph has to have at least one source since its orientation is acyclic. Pick one of these sources and denote it by $s_R$. We obtain that $s_R$ must be connected to the set of black vertices by some edges because $o_2$ has a unique source $s$ distinct from $s_R$ and because $s_R$ is a source in $\R^{r}_\Pi (o_2)$. Since $s_R$ is a source in $\R^{r}_\Pi (o_2)$, it follows that in $o_1$ all the red edges incident to $s_R$ have to point toward $s_R$ (and there is at least one such edge since $s_R$ is a red vertex connected by an edge to some black vertex). Hence, $f(o_1)_w \geqslant f(o_2)_w$ for all $w \in \overline{s}_R$, and there exists $w' \in \overline{s}_R$ such that the inequality is strict. Thus there exists $v \in \supp(\sigma) \cap \overline{s}_R$. Since $f(o_1)_w \geqslant f(o_2)_w$ for all $w \in \overline{s}_R$, we obtain that $\N(v) \cap \overline{s}_R \subset \supp(\sigma)$. Arguing inductively we obtain that $\overline{s}_R \subset \supp(\sigma)$. Take $s_B^{(1)}$ a black vertex connected to $s_R$ by an edge, and consider an undirected path of black vertices $\{s_B^{(i)}\}_{i = 1}^l$ such that $s = s_B^{(l)}$. A similar argument as the one given for $s_R$ gives that $\overline{s}_B^{(1)} \subset \supp(\sigma)$, and by induction we obtain that $\overline{s}_B^{(i)} \subset \supp(\sigma)$ for all $i \in \overline{1,l}$. Thus $\overline{s} = \overline{s}_B^{(l)}\subset \supp(\sigma)$. But we have proven that $\supp(\sigma) \cap \overline{s} = \emptyset$. We hence arrived at a contradiction and the proof of the proposition is complete. 
\end{proof}

\begin{rmk}
This proposition also shows that $\ff$ maps acyclic orientations on $\tG$ to $\B(\Pi)/ \img L$.
\end{rmk}

This result is encouraging since it indicates that it might be possible to find a correspondence between superstable configurations and divisors that contribute to the Betti numbers. This intuition is further supported by the next corollary, which suggests that boundary divisors given by a $(k+1)$-partition have non-trivial $(k-1)$-reduced homology groups over $\mathbb{C}$. 

\begin{corollary}
\label{prop2}
If $D_0$ is a $\Pi$-boundary divisor of $G$, then there is no effective divisor $D_1$ supported in all the connected components of $\Pi$ such that $D_0 \sim D_1$.
\end{corollary}

\begin{proof}
\label{support corollary}
Suppose there exists an effective divisor $D_1 \sim D_0$ that is supported in all the connected components of $\Pi$. By proposition \ref{prop1} we know that the map:
\[
\ff \colon \AUS \rightarrow \B(\Pi) / \img \Lap
\]
is bijective. Hence, there exists $o \in \AUS$ such that $f(o) \sim D_0$. Take $\sigma \geqslant 0$, $\sigma \not > 0$ a script such that $D_1 - \Lap \sigma = f(o)$. Since $\supp(D_1) \cap \overline{s} \neq \emptyset$, it follows that $\supp(\sigma) \cap \overline{s} \neq \emptyset$. Now we can repeat the argument given in the proof of proposition \ref{prop1} to show that $\supp(\sigma) \cap \overline{s} = \emptyset$. We thus obtain a contradiction and the claim is proven.

\end{proof}

We now turn to the case $k= n-1$, where $n$ is the number of vertices in the graph $G$. Wilmes proved the conjecture in this case. In order to state his result we recall what minimally alive divisors are. A divisor $D$ is said to be \emph{unstable} if there exists $v\in V$ such that $D_v \geqslant \deg(v)$, and it is said to be \emph{alive} if there is no stable divisor in $|D|$. Moreover, the divisor $D$ is \emph{minimally alive} if it is alive and if $D -1_v$ is not alive for all $v \in V$ \cite{PPW}. The statement of Wilmes' result is the following(\cite{W}, \cite{PPW}):

\begin{theorem}
The $(n-1)$-st Betti number is equal to the number of maximal superstable configurations on $G$. Furthermore, $\beta_{n-1,D} > 0$ if and only if $D$ is minimally alive.  
\end{theorem}

Now, it is immediate that if $k = n-1$, then $G_\Pi$ is isomorphic to $G$. This is so because there is exactly one partition $\Pi$ in $\Pcal_{n}$ and each of its connected components contains exactly one vertex. Let $D$ be a $\Pi$-boundary divisor and let us label $G$'s vertices with the integers $1$ through $n$. We check that $\Delta_D$ contains all the faces $[1,2,...,\hat{j},...,n]$. Let $v\in V$, then there exists at least one acyclic orientation of $G$ with unique source $v$. Moreover, by proposition \ref{prop1} there exists $o$ an acyclic orientation with unique source $v$ such that $\ff(o) \sim D$. But $\supp(\ff(o)) = V\setminus {v}$, and the claim follows. The set of the faces $[1,2,...,\hat{j},...,n]$ is a cycle because it represents the boundary of $[1,2,...,n]$. However, $[1,2,...,n]$ is not contained in $\Delta_D$ because of corollary \ref{prop2}. By Hochster's formula it follows that $\beta_{n-1,D} > 0$ for all $\Pi$-boundary divisors $D$. Again by proposition \ref{prop1} we obtain that $\beta_{n-1} \geqslant \# \{G \text{-maximal parking functions}\}$. Since we know that the inequality is in fact an equality, we obtain that:

\begin{proposition}
Let $\Pi$ be the unique $n$-partition of the graph $G$. Then, a divisor $D$ is $\Pi$-boundary for the partition $\Pi \in \Pcal_{n}$ if and only if it is minimally alive.
\end{proposition} 

In general, we expect the following two results about boundary divisors to be true:

\begin{conjecture}
\label{conj BD 1}
Let $D$ be a divisor of $G$. Then $\beta_{k,D} > 0$ if and only if $D \in \B(\Pi) / \img \Lap$ for some $\Pi \in \Pcal_{k+1}$.
\end{conjecture}

\begin{conjecture}
\label{conj BD 2}
Let $D \in \divisors(G)$ and suppose there exist exactly $l$ distinct partitions $\Pi_j \in \Pcal_{k+1}$ such that $D \in \B(\Pi_j)/\img \Lap$. Then $\beta_{k,D} = l$.
\end{conjecture}

These two statements were seen to be true for $k = 1$ and $k = n-1$, and it is easy to see that together these two results imply Wilmes' conjecture. 

In the next section we start building tools needed to study the problem for $G$ a multi-edged tree.

\section{Extension Cycles}
\label{extension cycles}

In this section we construct a special type of cycles of an abstract simplicial complex, called \emph{extension cycles}. This special type of cycles will be relevant in studying the homology of $\Delta_D$, where $D$ is a boundary divisor of a multi-edged tree. In this section we work in the simplicial complex $\Delta^{2k}$, which has a unique facet $[1,2,...,2k]$. We make this convention because we want to have all the $(k-1)$ dimensional faces with vertices in $\overline{1,2k}$ available. After we make it clear what we mean by an extension cycle, it will become evident that the same construction can be made in any simplicial complex that contains a certain collection of faces. We call the elements in $B = \{1,2,...,k\}$ \emph{base vertices}. We choose $E \subset \{k+1,k+2,...,2k\}$ and we refer to its elements as \emph{extension vertices}. We choose some faces $A$ and $A_j$ in $\Delta^{2k}$ such that:
\begin{align*}
A = [1,2,3,...,k]\\
A_j = [1,2,...,\hat{j},...,k,e_j],
\end{align*}
where $j\in \overline{1,k}$ and $e_j \in E$. Since label changing produces an isomorphic simplicial complex, we can assume for convenience that $e_i \geqslant e_j$ whenever $i \leqslant j$. We denote by $B_j$ the set of $A_j$'s base vertices and by $E_j$ the set of $A_j$'s extension vertices. We observe that $E_j$ is a singleton for every $j\in B$. In what follows we will often identify the base vertices with indices in the set $\{1,2,...,k\}$.

\begin{lemma}
Let $J \subset B$, then $\bigcap_{j\in J}B_j = B \setminus J$.
\end{lemma}

\begin{proof}
It is immediate to check that the two sets are included one into the other.  
\end{proof}

We make the convention that faces of the simplex $\Delta^{2k}$ written between brackets always have distinct vertices and the vertices will be written in ascending order. By the choice of base vertices and extension vertices it is immediate that the former will be clustered to the left, while the latter will be clustered to the right. 

We consider the collection of sets $\Jcal = \{J \subset B \colon \#\left(\bigcup_{j\in J} E_j \right) = \#(J) \}$. In particular, we have that $\Jcal$ contains the empty set and all the singleton subsets of $B$. The reader can easily notice that the sets in $\Jcal$ are simply those sets $J$ which have the property that if $j_1, j_2 \in J$ are two distinct elements, then $e_{j_1} \neq e_{j_2}$. 

For $J \in \Jcal$ we consider the faces
$
A_J = \left[\bigcap_{j\in J} B_j\right]\oplus \left[\bigcup_{j \in J} E_j \right]
$. By the lemma presented above and by the definition of $\Jcal$ it follows that the $A_J$'s are $(k-1)$-dimensional faces. We identify $A = A_\emptyset$ and $A_j = A_{\{j\}}$, for $j\in B$. We denote the collection of faces $A_J$ with $J \in \Jcal$ by $\Acal$. We prove that the faces in $\Acal$ form a cycle in $\Delta^{2k}$. This is what we refer to as $[1,2,...,k]$'s \emph{extension cycle} with \emph{extensions} $A_j$, $j\in B$. 

\begin{example}
Before we go into proving that the set $\Acal$ represents a cycle, let us consider a few extension cycles with base $[1,2,3]$:

\begin{figure}[ht]
\label{figure7}
\begin{center}

\begin{tikzpicture}
[scale = .4, very thick = 15mm]
  
   \node (n4) at (2,7)  [test] {};
  \node (n1) at (0,2) [test] {};
  \node (n2) at (3,1.2)  [test] {};
  \node (n3) at (6,2)  [test] {};
  \foreach \from/\to in {n1/n2,n1/n4, n3/n4,n2/n4,n2/n3}
    \draw[] (\from) -- (\to);
    \draw[dotted] (n1) -- (n3);
     
   
   \node (m4) at (1.3,7.5) [test4] {$4$};
    \node (m1) at (-0.5,2.5) [test4] {$1$};
    \node (m2) at (2.4,0.8) [test4] {$2$};
    \node (m3) at (6.5,2.5) [test4] {$3$};
    
    \node (n4) at (12,7)  [test] {};
  \node (n1) at (10,2) [test] {};
  \node (n2) at (13,1.2)  [test] {};
  \node (n3) at (16,2)  [test] {};
  \node (n5) at (15,7) [test] {};
  \foreach \from/\to in {n1/n2,n1/n4,n2/n4,n2/n3, n2/n5,n5/n3,n4/n5}
    \draw[] (\from) -- (\to);
    \draw[dotted] (n1) -- (n3);
    \draw[dotted] (n4) -- (n3);
    
    \node (m4) at (11.3,7.5) [test4] {$4$};
    \node (m1) at (9.5,2.5) [test4] {$1$};
    \node (m2) at (12.4,0.8) [test4] {$2$};
    \node (m3) at (16.5,2.5) [test4] {$3$};
    \node (m5) at (15.5, 7.5) [test4] {$5$};
    
    \node (n1) at (20,2) [test] {};
  \node (n2) at (23,1.2)  [test] {};
  \node (n3) at (26,2)  [test] {};
  \node (n4) at (20,7)  [test] {};
    \node (n5) at (23,6.2) [test] {};
    \node (n6) at (26,7)  [test] {};
    
    \foreach \from/\to in {n1/n2,n1/n4,n2/n4,n2/n3, n2/n5,n5/n3,n4/n5, n5/n6, n4/n6, n6/n3}
    \draw[] (\from) -- (\to);
    \draw[dotted] (n1) -- (n3);
    \draw[dotted] (n1) -- (n6);
    
    \node (m4) at (19.5,7.5) [test4] {$4$};
    \node (m1) at (19.5,2.5) [test4] {$1$};
    \node (m2) at (22.4,0.8) [test4] {$2$};
    \node (m3) at (26.5,2.5) [test4] {$3$};
    \node (m5) at (22.5, 5.7) [test4] {$5$};
    \node (m6) at (26.5,7.5) [test4] {$6$};
    
\end{tikzpicture}

\caption{A few extension cycles with base $[123]$.}
\end{center}
\end{figure}
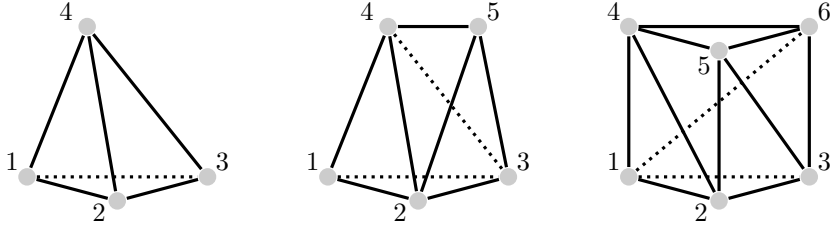 

\end{example}

The \emph{$l$-th layer} of $\Acal$ is the collection of faces $A_J$ with $J \in \Jcal$ such that $\# (J)= l$. In particular the $0$-th layer contains just the face $A$ and the $1$-st layer contains the faces $A_j$, $j\in B$. For the rest of this section we make the convention that $J$ is an element of $\Jcal$.

Before we get to the main result of this section we introduce a further convention. We use the usual definition of the boundary map for simplicial complexes such that $\partial[i_1,i_2,...,i_k] = \sum_{j = 1}^k (-1)^{j-1}[i_1,i_2,...,\hat{i_j},...,i_k]$, for $i_1 < i_2 < ... < i_k$ in $\{1,2,...,2k\}$. We call the $(k-2)$-dimensional faces that appear in the expansion of $\partial [i_1,i_2,...,i_k]$ the children of $A_J$. We denote the set of $[i_1,i_2,...,i_n]$'s children by $\supp\left(\partial [i_1,i_2,...,i_n]\right)$. We denote by $\Ecal = \bigcup_{J \in \Jcal} \supp\left(\partial A_J\right)$, the set of all children of $A_J$ when $J$ ranges over $\Jcal$.

\begin{proposition}
\label{boundary map}
For each $J \in \Jcal$ there exists a choice of sign $\epsilon_J \in \{-1,+1\}$ such that:
\begin{equation}
\label{bdm equation}
\partial \left(\sum_{J\in\Jcal} \epsilon_J A_J\right) = 0.
\end{equation}

\end{proposition}
\begin{proof}
We begin by constructing a choice of signs $\epsilon_J \in \{-1,+1\}$ such that equation \ref{bdm equation} is satisfied. Let $\epsilon_\emptyset = +1$ and let $\epsilon_j = (-1)^{j -1 +k}$, for $j\in B$. Now, for $J \in \Jcal \setminus \{\emptyset,\{1\},\{2\},...,\{k\}\}$, let $\epsilon_J = \prod_{j\in J}\epsilon_j$.

We show that for each $N \in \Ecal$ there exist exactly two distinct $J_1$ and $J_2$ in $\Jcal$ such that $N \in \supp(\partial A_{J_1}) \bigcap \supp(\partial A_{J_2})$, and that in the expansion of $\partial \left(\sum_{J\in\Jcal} \epsilon_J A_J\right)$ the two occurrences of $N$ have opposite signs. 

Let $N \in \Ecal$, by the definition of $\Ecal$ there must exist $J_1 \in \Jcal$ such that $N \in \supp(\partial A_{J_1})$. We want to show that there exist exactly one $J_2 \neq J_1$ in $\Jcal$ such that $N \in \supp(\partial A_{J_2})$ and that $N$ has distinct signs in $\epsilon_{J_1}\partial A_{J_1}$ and in $\epsilon_{J_2}\partial A_{J_2}$. We split the problem into two cases:

\case{Case 1: $N = \left[ \bigcap_{j\in J_1 \cup j_1} B_j \right]\oplus \left[\bigcup_{j \in J_1} E_j \right] \in \supp(\partial A_{J_1})$, for some $j_1 \in B\setminus J_1 \neq \emptyset$.}

It is immediate that $N$ is a well defined child of $A_{J_1}$. We need to study two subcases: $J_1 \cup j_1 \in \Jcal$ and $J_1 \cup j_1 \not \in \Jcal$. Let us start with the former case. Then, we can take $J_2 = J_1 \cup j_1$ and it is clear that $N \in \supp(\partial A_{J_2})$. Let us check the signs. Suppose $j_1$ is on the $t$-th position in $A_{J_1}$. Then $N$ has sign $(-1)^{t-1}\epsilon_{J_1}$ in $\epsilon_{J_1}\partial A_{J_1}$. By our assumption regarding the order of the $e_j$'s it follows that $e_{j_1}$ appears on the position $k-(j_1 - t)$ in $A_{J_2}$. Hence, $N$ has sign $(-1)^{k-j_1 + t-1}\epsilon_{J_2}$ in $\epsilon_{J_2}\partial A_{J_2}$. We need to verify that $(-1)^{t-1}\epsilon_{J_1}(-1)^{k-j_1 + t-1}\epsilon_{J_2} = -1$. But 
$(-1)^{t-1}\epsilon_{J_1}(-1)^{k-j_1 + t-1}\epsilon_{J_2} = (-1)^{k-j_1}\epsilon_{j_1} = (-1)^{k-j_1}(-1)^{j_1 - 1 + k} = -1.$

Let $J_3 \in \Jcal$ such that $N \in \supp(\partial A_{J_3})$ and let $l = \#(J_1)$. Then, $A_{J_3}$ must be in the $l$-th layer or in the $(l+1)$-st layer. Suppose it is in the $(l+1)$-st layer. Then $A_{J_3}$ contains $k-l-1$ base vertices, the same as $N$. By the lemma it follows that $B \setminus (J_1 \cup j_1) = B \setminus J_3$, which implies that $J_3 = J_1 \cup j_1 = J_2$. We now assume that $A_{J_3}$ is in the $l$-th layer. Then, $B\setminus (J_1 \cup j_1) \subset B\setminus J_3$. Thus,
$J_3$ is a subset of $J_1 \cup j_1$. Since $A_{J_3}$ and $N$ must have the same extension vertices and since $J_1 \cup j_1 \in \Jcal$, we obtain $J_3 = J_1$.

We now have to study what happens when $J_1 \cup j_1 \not\in \Jcal$. Since $J_1 \in \Jcal$, there exists $j_2 \in J_1$ such that $J_1 \cup j_1 \setminus j_2 \in \Jcal$. Take $J_2 = J_1 \cup j_1 \setminus j_2$, then $N \in \supp(\partial A_{J_2})$. Let us check that $N$ has opposite signs in $\epsilon_{J_1}A_{J_1}$ and $\epsilon_{J_2}A_{J_2}$. Suppose again that $j_1$ appears on position $t$ in $A_{J_1}$. By the order we assumed on the $e_j$'s we obtain that $j_2$ appears in $ A_{J_2}$ on position $t + j_2 - j_1$. Hence, we need $(-1)^{t-1}\epsilon_{J_1}(-1)^{t+j_2-j_1}\epsilon_{J_2} = -1$, but this is easy to check. A similar argument with the one presented above shows that there is no $J_3 \in \Jcal$ distinct from $J_1$ and $J_2$ such that $N \in \supp(\partial A_{J_3}).$

\case{Case 2: $N = \left[ \bigcap_{j\in J_1} B_j \right]\oplus \left[\bigcup_{j \in J_1 \setminus {j_1}} E_j \right] \in \supp(\partial A_{J_1})$, for some $j_1 \in J_1 \neq \emptyset$.}

It is clear that $N$ considered in this case is a child of $A_{J_1}$. Take $J_2 = J_1 \setminus j_1$, then it is immediate that $J_1$ and $J_2$ are distinct and that $N \in \supp(\partial A_{J_2})$. We can denote $J_1' = J_2$ and we are now in case 1, and this completes the proof of the proposition.
\end{proof}

We recall that the unique layer zero face is called the base of the extension cycle and that the faces in the first layer are called extensions. For symmetry, we give a name to the faces in the remaining layers. We call them \emph{roof faces}. 

\section{Orientation classes}
\label{orientation classes}

In this section we introduce an equivalence relation on $\AO$, the set of acyclic orientations of $\tG$. Through the use of the map $\ff \colon \AO \rightarrow \divisors(G)$, defined in section \ref{generalized bd}, and the equivalence relations on $\AO$ we will be able to better understand the structure of the simplical complexes associated to boundary divisors, and thus better understand their homology. 

Let us consider an acyclic orientation $o_1$ on $\tG$. We say that $A\subset \tV(o_1)$ is \emph{critical} with respect to $o_1$ if all the edges between $A$ and $\tV(o_1)\setminus A$ point away from $A$ or all point toward $A$. For $A$ a critical set, a \emph{switch} at $A$ of $o_1$ is an operation that changes the orientation of all the edges between $A$ and $\tV(o_1)\setminus A$ and preserves the orientation of the remaining edges. This process produces a new acyclic orientation $o_2 \in \AO$. If $A$ is critical with respect to $o_1$, we denote the $A$-switch of $o_1$ by $o_1 \xrightarrow{A} o_2$.   

\begin{dfn}
We say two orientations $o_1, o_2 \in \AO$ are equivalent and we denote it by $o_1 \sim o_2$ if there exists a sequence of critical sets and switches such that:
\[
o_1 \xrightarrow{A_1,A_2,...,A_n} o_2
\]

It is easy to check that $\sim$ does define an equivalence relation on $\AO$.
\end{dfn}

We notice that $o_1 \sim o_2$ implies $\ff(o_1) \sim \ff(o_2)$. This is because there exists a sequence $\{A_i\}_{i=1}^n$ of critical sets which switched take $o_1$ to $o_2$. Furthermore, we can choose the sets $A_i$ in such a way that before each switch the arrows that have to change orientation point toward $A_i$.  Then we take $\sigma$ a script such that $\sigma_v = \sum_{i=1}^n \chi_{A_i}(v)$. One can verify that $\ff(o_2) = \ff(o_1) - \Lap \sigma$. This discussion together with the proof of proposition \ref{prop1}, leads us to the following equivalence.

\begin{lemma}
Let $o_1 ,o_2\in \AO$. Then $\ff(o_1) \sim \ff(o_2)$ if and only if $o_1 \sim o_2$. 
\end{lemma}

In addition, by proposition \ref{prop1}, this immediately implies that for some fixed vertex $s\in \tV$ each class in $\AO / \sim$ contains exactly one acyclic orientation with unique source $s$. This fact was already known in the literature \cite{BC}.

\section{Multi-edged Trees}

In this section we study the homology of boundary divisors when the graph $G$ is a multi-edged tree. The main result of this section is the following: 

\begin{proposition}
\label{hom prop 1}
Let $G$ be a multi-edged tree and let $D \in \divisors(G)$ be a boundary divisor with respect to the partition $\Pi \in \Pcal_{k+1}$. Then $\beta_{k,D} > 0$.
\end{proposition} 

In order to be able to study the homology of these boundary divisors we need to define several more concepts. We start by introducing the "boundaries" of the connected components of a partition $\Pi \in \Pcal_{k+1}$. Suppose $\{\pi_1, \pi_2,..., \pi_{k+1}\}$ are these connected components; we define: 

\[
B(\pi_j,I) = \left\{v \in \pi_j \colon \sum_{\pi_i \in I} deg_{\overline{\pi}_j\overline{\pi}_i}(v) > 0 \right\}
\] 
for $I \subset \{\pi_1,\pi_2,...,\pi_{k+1}\} \setminus \{j\}$. If $I = \{\pi_1,\pi_2,...,\pi_{k+1}\} \setminus \{\pi_j\}$, we denote $B_{j} = B(\pi_j,I)$. The set $B_j$ represents the \emph{boundary} of $\pi_j$. By $G_b = (V_b, E_b)$ we denote the graph induced by $G$ on the set $V_b = \bigcup_{j=1}^{k+1} B_j$. We call the set $V_b$ the \emph{boundary} of $\Pi$. For a partition $\Pi\in\Pcal_{k+1}$ and a boundary divisor $D \in \B(\Pi)$ we define the \emph{orientation class} of $D$ on $\tG$ as the set:
\[
\Oo(D) = \left\{o \in \AO \colon f(o) \sim D \right\}.
\]

By our discussion regarding orientation classes from the previous section we can say that $\Oo(D)$ represents exactly one class in $\AO / \sim$. Let $\Delta_D^{(b)} = \{W \colon W \subset \supp(f(o)) \text{, for some } o \in \Oo(D)\}$. Then $\Delta_D^{(b)}$ is a subcomplex of $\Delta_D$, supported in $V_b$. We refer to this complex as $D$'s \emph{boundary simplicial complex}.

In order to prove proposition \ref{hom prop 1} we will construct 
a cycle in the $(k-1)$-st chain group for the simplicial complex $\Delta_D^{(b)}$ and then show that it is not spanned by boundaries in $\Delta_D$. For this we need the notion of $k$-\emph{supported} faces. Let $W \in K$, where $K$ is a simplicial complex on the vertices of $G$. We say $W$ is $k$-supported with respect to a partition $\Pi\in\Pcal_{k+1}$ if its support intersects exactly $k$ of the $k+1$ connected components of $\Pi$. Recall that we denoted by $\{\pi_1, \pi_2,..., \pi_{k+1}\}$ the connected components of $\Pi$. If $W$ is $k$-supported in $\Pi \setminus \pi_1$ and is $(k-1)$-dimensional, we say $W$ is $\pi_1$-\emph{essential}. The following lemma is crucial in the proof of proposition \ref{hom prop 1}.

\begin{lemma}
\label{essential lemma}
Let $D$ be a $\Pi$-boundary divisor and let $W \in \Delta_D$ be a $\pi_1$-essential face. Let $\{A_i\}_{i=1}^l$ be a set of $k$-dimensional faces in $\Delta_D$ and let $\{c_i\}_{i=1}^l$ be a sequence of complex numbers such that $W \in \supp\left( \partial\left( \sum_{i=1}^l c_iA_i\right)\right)$. Then, there exists $W_0 \neq W$ a $\pi_1$-essential face such that $W_0 \in \supp\left(\partial\left(\sum_{i=1}^l c_iA_i\right)\right)$.
\end{lemma}

\begin{proof}
We can assume without loss of generality that $c_i \in \mathbb{R}$, for $i =\overline{1,l}$. This is because we can split the sum $\sum_{i=1}^l c_iA_i$ into real and imaginary parts and work with each of them individually. By corollary \ref{prop2}, discussed in section \ref{generalized bd}, we obtain that it is enough to study the case where each $A_i$ is $k$-supported in $\Pi\setminus \pi_1$, for $i\in \overline{1,l}$. 

We use the partition $\Pi$ to choose the positive orientation for all $\pi_1$-essential faces in $\Delta_D$. If $A$ is $\pi_1$-essential, then we write $A = a_2\wedge a_3 \wedge...\wedge a_{k+1}$, where $a_i \in \pi_i$ for $i\in \overline{2,k+1}$. We can choose such an orientation when $A$ is $\pi_1$-essential because $\# (\supp (A) \cap \pi_i) = 1$ for $i \in \overline{2,k+1}$.  We now write $A_i = a_{i1}\wedge a_{i2} \wedge ... \wedge a_{i(k+1)}$ with the property that there exists $r\in \overline{2,k}$ such that $a_{ij} \in \pi_{j+1}$ for $j\in \overline{1,r}$ and $a_{ij}\in \pi_{j}$ for $j\in\overline{r+1,k+1}$. This can be done since $A_i$ is $\pi_1$-essential and has $k+1$ vertices  supported in $k$ connected components. From this choice of orientation for the $A_i$'s it is obvious that there are exactly two $\pi_1$-essential faces in $\supp\left(\partial A_i\right)$, and that they appear with opposite signs. For each face $A_i$, we denote by $B_{i1}$ and $B_{i2}$ the $\pi_1$-essential faces that belong to $\supp(\partial A_i)$. 

We introduce the following "characteristic" function:

\[
\chi_{>0}(B,A) = 
\begin{cases}
	1, 		&\text{if $B$ has positive sign in $\partial A$.}\\
	0,     &\text{otherwise}
\end{cases}
\]

The function $\chi_{<0}$ is defined analogously (it is equal to $1$ if $B$ has negative sign in $\partial A$). We now consider the following sums:

\begin{align*}
S_+ = \sum_{i =1}^l\left(\chi_{> 0}(B_{i1}, A_i)c_i + \chi_{> 0}(B_{i2}, A_i)c_i\right) \\
S_- = \sum_{i =1}^l\left(\chi_{< 0}(B_{i1}, A_i)c_i + \chi_{< 0}(B_{i2}, A_i)c_i\right).
\end{align*}
Since for each face $A_i$, $B_{i1}$ and $B_{i2}$ appear with distinct signs in the expansion of $\partial A_i$, it follows that $S_+ = S_-$. 

Suppose there is no $W_0 \neq W$ such that 
$W_0 \in\supp\left(\partial\left(\sum_{i=1}^l c_iA_i\right)\right)$. Then the contributions of each $\pi_1$-essential face in $\bigcup_{i=1}^l \supp(\partial A_i) \setminus \{W\}$ to $S_+$ and $S_-$ are equal. Since $S_+ = S_-$, we obtain that $W$'s contributions to $S_+$ and $S_-$ are equal, as well. It follows that $W \not \in \supp\left(\partial\left(\sum_{i=1}^l c_iA_i\right)\right)$. This is a contradiction, and thus the lemma is proven.
\end{proof}

The idea of the proof of proposition \ref{hom prop 1} is to construct an extension cycle that contains exactly one $\pi_1$-essential face. In order to give a clear presentation of how to achieve this, we need another helping map. We introduce the map $\T \colon \Pi \rightarrow \Pi$, from the set of $\Pi$'s connected components to itself. The map $\T$ is defined as follows:

\[
\T(\pi_i) = 
\begin{cases}
	\pi_1, 		&\text{if $\pi_i = \pi_1$.}\\
	\pi_j,     &\text{$\pi_j \in \Pi$, $\dist(\pi_i, \pi_j) = 1$ and $\dist(\pi_1, \pi_j) = \dist(\pi_1,\pi_i) - 1$.}
\end{cases}
\]

Since $G$ is a multi-edged tree, this map is well defined. For brevity, we might write $\T(i)$ instead of $\T(\pi_i)$. We illustrate the behaviour of the map $\T$ and an example of an extension cycle.

\begin{example}
\label{example7}
We consider the multi-edged tree $G$ in Figure $10$. We denote the connected components of the partition presented in the picture as follows: $\pi_1 = \{1\}$, $\pi_2 = \{2, 3\}$, $\pi_3 = \{4\}$, $\pi_4 = \{5\}$ and $\pi_5  = \{6\}$. In this example, since $\tG$ is a tree, all the orientations of $\tG$ belong to the same equivalence class as $o$. Using this it can be checked that there is an extension cycle in $\Delta_D$ with base $[2,4,5,6]$ and extensions $[2,3,4,5]$, $[2,3,4,6]$, $[1,2,5,6]$ and $[1,4,5,6]$. By the construction of an extension cycle described in section \ref{extension cycles}, it follows that, in addition to the faces already mentioned, the cycle must also contain the following roof faces: $[1,2,3,5]$, $[1,3,4,5]$, $[1,2,3,6]$ and $[1,3,4,6]$. By inspection, one sees that these faces are contained in $\Delta_D$. For this cycle, the extension vertices are $\{1,3\}$. 
For this particular multi-edged tree $G$, we have $T(\pi_1) = \pi_1$, $\T(\pi_2) = \pi_1$, $\T(\pi_3) = \pi_2$, $\T(\pi_4) = \pi_2$ and $\T(\pi_5) = \pi_4$.

\begin{figure}[ht]
\label{figure8}
\begin{center}

\begin{tikzpicture}
[scale = .4, very thick = 15mm]
  
   \node (n4) at (1.5,4)  [test] {$1$};
  \node (n1) at (3,10) [test] {$0$};
  \node (n2) at (3,6)  [test] {$1$};
  \node (n3) at (6,6)  [test] {$0$};
  \node (n5) at (8,4) [test] {$1$};
  \node (n6) at (8, 1) [test] {$1$};
  \foreach \from/\to in {n1/n2,n2/n3, n2/n4,n3/n5,n5/n6}
    \draw[] (\from) -- (\to);
    \path[] (n2) edge [bend left = 20] node {} (n3);
    \path[] (n2) edge [bend right = 20] node {} (n3);
   
   \node (m4) at (0.8,4.6) [test4] {$4$};
    \node (m1) at (2.3,10.6) [test4] {$1$};
    \node (m2) at (2.3,6.6) [test4] {$2$};
    \node (m3) at (6.7,6.6) [test4] {$3$};
    \node (m5) at (8.7,4.6) [test4] {$5$};
    \node (m6) at (8.7,1.6) [test4] {$6$};
    \node (m7) at (1, 11) [test2] {$G$};
    
    \draw[dotted] (2,8) -- (4,8);
    \draw[dotted] (1.7,5.5) -- (3, 4.3); 
    \draw[dotted] (6.7,4.4) -- (8,5.8);
    \draw[dotted] (7,2.5) -- (9,2.5);
    
     \node (n4) at (16.5,4)  [test] {};
  \node (n1) at (18,10) [test] {};
  \node (n2) at (18,6)  [test] {};
  \node (n3) at (18,6)  [test] {};
  \node (n5) at (20,4) [test] {};
  \node (n6) at (20, 1) [test] {};

  \path[->] (n1) edge [right] node {$e_1$} (n2) ;
	\path[->] (n2) edge [right] node {$e_2$} (n4) ;
	\path[->] (n3) edge [right] node {$e_3$} (n5) ;
	\path[->] (n5) edge [left] node {$e_4$} (n6) ;  
   
   \node (m4) at (15.8,4.6) [test4] {$4$};
    \node (m1) at (17.3,10.6) [test4] {$1$};
    \node (m2) at (17,6.6) [test4] {$2,3$};
    
    \node (m5) at (20.7,4.6) [test4] {$5$};
    \node (m6) at (20.7,1.6) [test4] {$6$};
    \node (m7) at (15, 11) [test2] {$\tG (o)$};
    
\end{tikzpicture}

\caption{A boundary divisor $D$ for the multi-edged tree $G$ and the corresponding $\tG (o)$ graph.}
\end{center}
\end{figure}
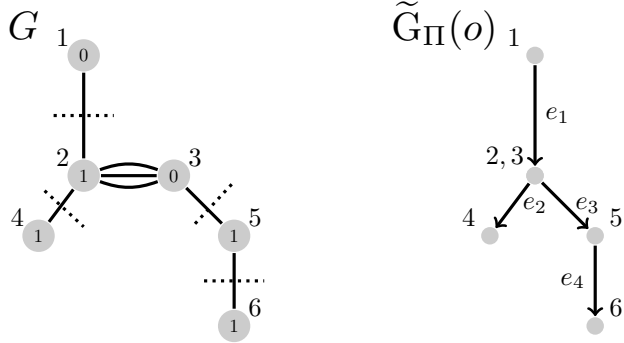

\end{example}

We are now ready to prove the main result of this section.

\begin{proof}[Proof of Proposition \ref{hom prop 1}]
We prove that $\beta_{k,D} > 0$ by constructing a special type of $(k-1)$-dimensional cycle and showing that it is not in $\img \partial_{k}$.

Let $\{\pi_1,\pi_2,...,\pi_{k+1}\}$ be $\Pi$'s connected components. Assume we constructed a $(k-1)$-cycle in $\Delta_D$ that contains exactly one $\pi_1$-essential face. If this cycle were in $\img \partial_k$, then by lemma \ref{essential lemma}, presented above, we would obtain that there has to be a second $\pi_1$-essential face in the cycle. Thus, after we construct a $(k-1)$-cycle in $\Delta_D$ that contains exactly one $\pi_1$-essential face, we would know that the $(k-1)$-st reduced homology group of $\Delta_D$ is not trivial. Hence, we would get $\beta_{k,D} > 0$.  

We turn to the construction of such a cycle in the boundary simplicial complex $\Delta_D^{(b)}$. Take $o_1$ to be the unique orientation with unique source $s = \pi_1$. By proposition \ref{prop1} we obtain that $\ff (o_1) \sim D$. Now, from each of the boundaries $\B(\pi_j, T(j))$ ($1 < j \leqslant k+1$) we choose one vertex $b_j$. Because $B = \{b_2,b_3,...,b_{k+1}\} \subset \ff(o_1)$, we obtain that the face $[b_2,b_3,...,b_{k+1}]$ is contained in $\Delta_D^{(b)}$. Moreover, this is a $\pi_1$-essential face. We construct an extension cycle with base $[b_2,b_3,...,b_{k+1}]$.

We now need to choose extension vertices. We want to choose appropriate $e_j$'s such that the faces $[b_2,..., \hat{b_j},...,b_{k+1},e_j]$ are the extensions of our cycle. We pick $e_j$ in the following way. Let $m \in \mathbb{N}$ minimal such that 

\[
B(\T^{(m+1)}(j), \T^{(m)}(j)) \setminus B \not = \emptyset.
\]  
Here, by $T^{(m)}$ we mean $T$ composed with itself $m$ times. If $m = 0$ we consider the identity map. To check that such a $m$ exists note that $\pi_1 \cap \supp(\ff (o_1)) = \emptyset$. Now, we choose any $e_j$ in $B(\T^{(m+1)}(j), \T^{(m)}(j)) \setminus B$. Consider the orientation $o_j$ that changes the orientation between $\T^{(l)}(j)$ and $\T^{(l+1)}(j)$ for $0 \leqslant l \leqslant m$, and preserves the orientation given by $o_1$ to the remaining edges. Then, by the minimality of $m$, we obtain $\{b_2,..., \hat{b_j},...,b_{k+1},e_j\} \subset \supp( \ff(o_j))$. Hence, $[b_2,..., \hat{b_j},...,b_{k+1},e_j] \in \Delta_D^{(b)}$. Recalling that for a multi-edged tree all the orientations on $\tG$ belong to the same equivalence class, the reader can check that the roof faces needed for this extension cycle are present in $\Delta_D^{(b)}$ and that the base face remains the only $\pi_1$-essential face in this cycle. In particular, the reader can verify that for the multi-edged tree presented in Example $9$, the construction described here determines uniquely the cycle mentioned there. Nonetheless, it should be clear that in general this construction produces an extension cycle, but is not necessarily unique. 

This construction shows that $\beta_{k,D} > 0$ when $D$ is a boundary divisor and $G$ is a multi-edged tree.  
\end{proof}

\section*{Acknowledgements}
This work was done as part of the REU program at University of Minnesota, Twin Cities, funded by the NSF grants DMS-1148634 and DMS-1001933. 
The author would like to thank professors Gregg Musiker, Vic Reiner, Dennis Stanton and Pavlo Pylyavskyy for mentoring this combinatorics program. The author is deeply grateful to Prof. Gregg Musiker for introducing him to the problem and for providing invaluable guidance, suggestions and help, and to Prof. Vic Reiner for fruitful discussions and for pointing out reference \cite{BC}. The author would also like to offer special thanks to Alexander Garver for contributing to the clarity of the exposition and for proofreading the material.

\noindent {{\textup{Princeton University, Department of Mathematics, Princeton, NJ 08544, USA}}}

\noindent \emph{E-mail address}: \textbf{hmania@princeton.edu}

\end{document}